\documentclass[3p,twocolumn]{elsarticle}

\usepackage{amsmath,amsthm,amsfonts,amssymb}
\usepackage{mathtools}
\allowdisplaybreaks
\usepackage{graphicx}
\usepackage{booktabs}
\usepackage{xcolor}
\usepackage{enumerate}
\usepackage{microtype}
\usepackage{tikz}
\usetikzlibrary{arrows.meta, positioning, backgrounds, fit, calc}
\definecolor{mDarkTeal}{RGB}{48,115,116}
\definecolor{RoyalPurp}{RGB}{120,81,169}
\definecolor{BoxFillColor}{RGB}{240,240,240}
\newcommand{\He}{\mathrm{He}}
\newcommand{\diag}{\mathrm{diag}}
\newcommand{\R}{\mathbb{R}}
\DeclareMathOperator{\sign}{sign}
\newcommand{\spow}[1]{\lceil#1\rfloor}  
\newtheorem{theorem}{Theorem}
\newtheorem{proposition}[theorem]{Proposition}

\newtheorem{corollary}[theorem]{Corollary}
\newtheorem{assumption}[theorem]{Assumption}
\newtheorem{remark}[theorem]{Remark}

\begin{document}

\begin{frontmatter}

\title{Observer design for Lur'e-type systems via injection of a reconstructed nonlinear output}

\author[inst1]{Adel Malik ANNABI}\ead{adel-malik.annabi@inria.fr}

\address[inst1]{INRIA, CNRS, Universit\'e Côte d'Azur, France}


\begin{abstract}
Observer design for Lur'e-type systems
typically reduces to solving a linear matrix inequality (LMI). In certain
cases, the observer gain may grow unbounded with the coupling norm.
We propose to reconstruct key nonlinear terms via a family of parallel
second-order sliding-mode observers and to inject them into a Luenberger
observer as additional measurements.  This feeds a second correction channel that
attenuates the nonlinear coupling in the error dynamics.  The convergence of
the combined observer is guaranteed under a new LMI which contains the classical
one as a special case.  We exhibit parameter regimes where classical designs
require impractically large gains, while the proposed design maintains
moderate gains.
A stability analysis of the proposed observer characterizes
the trade-off between the two designs and identifies the noise regimes
in which the proposed design is preferable.
We illustrate the approach on a Wilson--Cowan network showing the trade-off between the nonlinear coupling norm, observer gain, and noise.
\end{abstract}

\end{frontmatter}

\section{Introduction}
\label{sec:intro}

Estimating the internal state of a dynamical system from partial measurements
is a fundamental problem in control theory.  A particularly structured and
widely studied class of nonlinear systems is that of Lur'e-type systems: a linear
block in feedback interconnection with a static, memoryless nonlinearity.
Lur'e-type models arise naturally in many application domains, including
neuroscience~\cite{Wilson1972,Amari1977,Jansen1995,breakspear2017dynamic}, where the state is
typically only partially measured and reconstructing the full activity
vector is essential for monitoring and closed-loop control.

For Lur'e-type systems, the standard approach to state estimation is the
Luenberger observer, which copies the plant dynamics and corrects the
state prediction with a term proportional to the output prediction error.
When the nonlinearity satisfies an incremental sector or slope-restricted
condition, the convergence of the estimation error can be certified
through a linear matrix inequality (LMI), solvable by semidefinite
programming~\cite{Boyd1994,Arcak2001,Fan2003,Zemouche2013,giaccagli2023lmi}.  This LMI-based design
has been successfully applied to neural mass
models~\cite{Chong2012,Liu2014} for example. 

A known difficulty is that the LMI becomes increasingly stringent when the
nonlinear term is large compared to the linear dynamics.  When the nonlinear
feedback dominates, the LMI may become infeasible under any reasonable bound on
the observer gain.  The solver is then forced either to return gains that
are too large to be deployed in practice, or to declare infeasibility
altogether.  In either case, no usable Luenberger observer is produced.

As this obstruction is fundamentally a lack of measurements, the core idea of this work is that if we had access to key nonlinear coupling terms, one could inject them in the Luenberger observer and relax the LMI accordingly.  We show that, for the considered class of Lur'e-type systems, these signals can be reconstructed from the output itself: each output channel yields a scalar triangular dynamics from which the coupling term is recovered in finite time by a family of second-order sliding-mode observers~\cite{Levant2003,Moreno2008,bernard2017observers}, producing an additional \emph{virtual measurement}.

With this virtual measurement, we augment the classical Luenberger
observer with a second correction channel that compares the predicted
nonlinear output against the reconstructed signal.  The convergence of the
resulting combined observer is guaranteed by a new LMI that contains the
classical one as a special case.  Because the second channel attenuates
the effective nonlinear coupling before it reaches the error dynamics, the
new LMI remains feasible and yields moderate gains in regimes where
the classical one either fails or demands gains orders of magnitude
larger. However, the approach introduces the cost of the reconstruction of the virtual
measurement, which requires a family of sliding-mode observers and contributes its own
sensitivity to measurement noise.  In this regard,
we provide a
stability analysis that quantifies the steady-state error under
bounded measurement noise for both the classical and the proposed
observer.  This analysis delineates the noise regimes in which each design
is preferable and shows that the trade-off, that is accepting the
sliding-mode observer's noise footprint in exchange for a drastically reduced
observer gain, is favourable across a broad range of operating
conditions.

The paper is organized as follows.  Section~\ref{sec:problem} states the
system class and the sector condition.  Section~\ref{sec:observer} presents
the combined observer and the main convergence theorem with its Lyapunov-based
proof.  Section~\ref{sec:comparison} establishes the feasibility inclusion
between the classical and proposed LMIs and exhibits a closed-form gain that
attenuates the coupling, together with the structural case where the coupling
is removed from the error dynamics altogether.  Section~\ref{sec:iss} derives
the steady-state error bound under measurement noise and model disturbance.  Section~\ref{sec:simulations} provides a numerical comparison on a
Wilson--Cowan network, illustrating the gain reduction and the
noise trade-off.  Section~\ref{sec:conclusion} concludes.

\section{Problem statement}
\label{sec:problem}

\paragraph{Notation.}
Throughout, $\|\cdot\|$ denotes the Euclidean norm for vectors and the
induced $2$-norm for matrices, and $\|x\|_P := \sqrt{x^\top P x}$ the
$P$-weighted norm for $P \succ 0$.  For a square matrix $X$,
$\He\{X\} := X + X^\top$, and $\diag(\cdot)$ builds the diagonal matrix
with the listed entries.

Consider the Lur'e-type system
\begin{equation}
  \label{eq:system}
  \begin{cases}
    \dot{x} = Ax + Wf(x) + Bu + p(t),\\
    y = Cx + \nu(t),
  \end{cases}
\end{equation}
where $x \in \R^n$ is the state, $u \in \R^{n_u}$ the input, $y \in \R^{n_y}$
the measured output, and $p(t)$, $\nu(t)$ are bounded disturbances with
$\|p(t)\|\leq\bar{p}$, $\|\nu(t)\|\leq\bar{\nu}$ uniformly in $t$.
The matrices $A \in \R^{n\times n}$, $W \in \R^{n\times
n}$, $B \in \R^{n\times n_u}$, $C \in \R^{n_y\times n}$ are known.

The nonlinearity $f:\R^n\to\R^{n}$ is a $C^1(\R^n,\R^n)$ and Lipschitz continuous function. 
When $u\equiv 0$ and $f(0)=0$, the weaker assumption that $f$ is
component-wise H\"older continuous with exponent $\alpha\ge 1/2$
suffices (Remark~\ref{rem:holder}).

It satisfies the following incremental sector condition: there exist diagonal matrices
$\Gamma = \diag(\gamma_1,\ldots,\gamma_n)\succeq 0$ and
$\Lambda = \diag(\lambda_1,\ldots,\lambda_n)\succ 0$ such that, for all
$x,\hat{x}\in\R^n$, setting $\delta = f(\hat{x})-f(x)$ and $e = \hat{x}-x$,
\begin{equation}
  \label{eq:sector}
  \sum_{i=1}^{n} \lambda_i\,\delta_i\bigl(\delta_i - \gamma_i\,e_i\bigr) \leq 0.
\end{equation}
Two classical conditions that guarantee \eqref{eq:sector}
\cite{Arcak2001,Fan2003} are recalled below.
\begin{enumerate}[(i)]
  \item \emph{Component-wise slope bound.}  If $f$ acts component-wise and
        each $f_i$ satisfies $0 \leq f_i'(\xi)\leq \gamma_i$ for all
        $\xi\in\R$~\cite{Fan2003}, then \eqref{eq:sector} holds for any
        diagonal $\Lambda\succ 0$.
  \item \emph{Global operator bound.}  If $f$ has a symmetric Jacobian with
        $0\preceq J_f(x)\preceq \gamma I$ for all $x$, the integral mean-value
        theorem gives $\delta = \bar{J}e$ with $0\preceq\bar{J}\preceq\gamma I$,
        so $\delta^\top\lambda(\delta-\gamma e)\leq 0$; \eqref{eq:sector}
        holds with $\Gamma=\gamma I$ and scalar $\Lambda=\lambda I$.
\end{enumerate}

We also consider the weaker case where only component-wise monotonicity
$\delta_i e_i \geq 0$ is known (no slope bound $\gamma_i$ is available);
the corresponding increasing-only LMI is derived in
Remark~\ref{rem:increasing}.

\begin{assumption}\label{ass:compact}
  The input $u$ belongs to $L^2([0,\infty),\R^{n_u})$. 
  Solutions of \eqref{eq:system} are bounded in a known compact set $\mathcal{K}
  \subset \R^n$ for all $u\in\mathcal{U}$.
\end{assumption}
This assumption is standard in the literature. If, for example, the matrix $A$ is Hurwitz and the nonlinearity $f$ and the control $u$ are uniformly bounded, then it holds.

\medskip
\noindent\textbf{Observer design problem.}
Construct an auxiliary system driven by $(y, u)$ whose state $\hat{x}$
satisfies $\|\hat{x}(t) - x(t)\| \to 0$ exponentially for all initial
conditions $\hat{x}(0)$ in a given set.

\section{Combined observer and convergence theorem}
\label{sec:observer}

The proposed observer uses both the measured output $y = Cx$ and the signal
$y_2 := CAx + CWf(x)$, which is reconstructed in finite time from $y$.  The full
system is
\begin{equation}
  \label{eq:observer}
  \begin{cases}
    \dot{\hat{z}}_{1,j} = \hat{z}_{2,j} + [CBu]_j
      + L_j k_1\spow{y_j - \hat{z}_{1,j}}^{1/2}, \\[3pt]
    \dot{\hat{z}}_{2,j} \in  L_j^2 k_2\,\sign(y_j - \hat{z}_{1,j}),
      \qquad j = 1,\ldots,n_y, \\[5pt]
    \begin{aligned}
      \dot{\hat{x}} &= A\hat{x} + Wf(\hat{\zeta})
        + K(y - C\hat{x}) \\
      &\quad - K'\bigl(CA\hat{x} + CWf(\hat{\zeta})
        - \hat{z}_2\bigr) + Bu,
    \end{aligned}
  \end{cases}
\end{equation}
where $\hat{\zeta} = \hat{x} + E(y - C\hat{x})$, $\spow{\cdot}^{1/2}
= |\cdot|^{1/2}\sign(\cdot)$, and $K, K' \in \R^{n\times n_y}$,
$E \in \R^{n\times n_y}$, $L_j > 0$, $k_1,k_2 > 0$ are design parameters
(the gains $k_1,k_2$ are standard for the homogeneous
observer; see~\cite{Levant2003}).
The matrix $E$ shifts the argument of~$f$ inside the sector, reducing
conservatism~\cite{Fan2003}.

The upper block is a family of $n_y$ parallel second-order sliding-mode observers (solutions understood
in the Filippov sense~\cite{Levant2003}) that reconstruct $y_2 = CAx + CWf(x)$
from $y$; the lower block is the Lur'e-type injection observer that uses both
$y$ and the estimate $\hat{z}_2$.
Figure~\ref{fig:combined_observer} shows the resulting architecture and
contrasts it with the standard design.

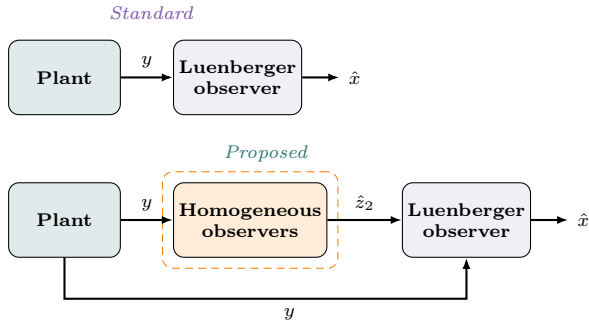
\begin{figure}[ht]
  \centering
  \resizebox{\columnwidth}{!}{%
  \begin{tikzpicture}[
    node distance=0.45cm and 0.7cm,
    block/.style = {draw, fill=BoxFillColor, rectangle, minimum height=1.0cm, minimum width=1.5cm, rounded corners, font=\scriptsize, align=center, inner sep=2pt},
    arrow/.style = {draw, -{Latex[length=1.5mm]}, thick},
    label/.style = {font=\scriptsize\itshape},
  ]
    \node [block, fill=mDarkTeal!15] (sys1) {\textbf{Plant}};
    \node [block, fill=RoyalPurp!10, right=0.7cm of sys1] (luen1) {\textbf{Luenberger}\\ \textbf{observer}};
    \node [right=0.5cm of luen1, font=\scriptsize\bfseries] (xhat1) {$\hat{x}$};
    \draw [arrow] (sys1) -- node[above, font=\scriptsize] {$y$} (luen1);
    \draw [arrow] (luen1) -- (xhat1);
    \node [label, RoyalPurp, above=0.15cm of $(sys1.north)!0.5!(luen1.north)$] {Standard};

    \node [block, fill=mDarkTeal!15, below=0.9cm of sys1] (sys2) {\textbf{Plant}};
    \node [block, fill=orange!15, right=0.7cm of sys2] (homo) {\textbf{Homogeneous}\\ \textbf{observers}};
    \node [block, fill=RoyalPurp!10, right=1.0cm of homo] (luen2) {\textbf{Luenberger}\\ \textbf{observer}};
    \node [right=0.5cm of luen2, font=\scriptsize\bfseries] (xhat2) {$\hat{x}$};

    \draw [arrow] (sys2) -- node[above, font=\scriptsize] {$y$} (homo);
    \draw [arrow] (sys2.south) -- ++(0,-0.55) -| node[font=\scriptsize, below, pos=0.28] {$y$} (luen2.south);
    \draw [arrow] (homo) -- node[above, font=\scriptsize] {$\hat{z}_2$} (luen2);
    \draw [arrow] (luen2) -- (xhat2);
    \node [label, mDarkTeal, above=0.15cm of $(sys2.north)!0.5!(luen2.north)$] {Proposed};

    \begin{scope}[on background layer]
      \node [draw=orange, densely dashed, rounded corners, inner sep=4pt,
             fit=(homo),
             label={[font=\scriptsize, orange]north west:{Reconstructs $\hat{z}_2$}}] {};
    \end{scope}
  \end{tikzpicture}%
  }
  \caption{Combined observer architecture.  \emph{Top:} the standard Luenberger observer driven by the measurement $y$ only.  \emph{Bottom:} the proposed design augments it with a family of parallel homogeneous observers that reconstruct the virtual output $\hat{z}_2 = CAx + CWf(x)$ from $y$, which is injected through the additional correction channel $K'$ together with the measurement $y$.}
  \label{fig:combined_observer}
\end{figure}

Throughout, we use the shorthands
\begin{equation}
  \label{eq:shorthand}
  \begin{aligned}
    A_{K'} &:= (I-K'C)A, \\
    W_{K'} &:= (I-K'C)W,
  \end{aligned}
\end{equation}
so that $A_{K'}=A$ and $W_{K'}=W$ when $K'=0$: the subscript records that
these are the system matrices left after the injection $K'$ removes the
part of the dynamics visible through the output.  Since $\Gamma$ is
diagonal, $\Gamma^\top=\Gamma$.

\begin{remark}[Why $E=0$ in the analysis and in the simulations]
  \label{rem:Ezero}
  When $E=0$, the matrix inequality~\eqref{eq:LMI} below is linear in the
  decision variables $(P,PK,PK',\Lambda)$ and can be solved by standard
  semidefinite programming.  With $E\neq0$, the product $E\Lambda$ makes the
  condition bilinear, requiring iterative methods~\cite{Fan2003}.
  Since the coupling attenuation acts on $W_{K'}$ through $K'$ alone and does
  not rely on $E$, we set $E=0$ in the stability analysis of
  Section~\ref{sec:iss} and in all numerical simulations
  (Section~\ref{sec:simulations}).
\end{remark}

\subsection{Design steps (heuristics)}
\begin{enumerate}
  \item Defining $z_{1,j}=[Cx]_j$ and $z_{2,j}=[CAx + CWf(x)]_j$, the dynamics
        per channel satisfy $\dot{z}_{1,j} = z_{2,j} + [CBu]_j$, a scalar
        triangular form~\cite{bernard2017observers} with known feedforward $[CBu]_j$.
        Under Assumption~\ref{ass:compact}, $\dot{z}_{2,j}$ is bounded; a
        homogeneous observer~\cite{Levant2003,Moreno2008} then recovers $z_{2,j}$
        from $z_{1,j}$ in prescribed finite time $T_0$
        for sufficiently large $L_j$~\cite{Moreno2008,bernard2017observers}.
        The observer acts as a robust exact reconstructor and delivers
        $\hat{z}_2 = z_2$.
        For $A = -a I$ (Wilson--Cowan; Section~\ref{sec:simulations}),
        we have $z_2 = CWf(x)$ only, reducing the required gain $L_j$.
  \item Once $\hat{z}_2(t) = y_2(t)$ for $t\geq T_0$, the lower observer is
        driven by both $y$ and the exact $y_2$.  With $e = \hat{x}-x$,
        $\hat{\zeta}-x = (I-EC)e$, and $\delta = f(\hat{\zeta})-f(x)$, the
        error dynamics become
        \begin{equation}
          \label{eq:error}
          \dot{e} = \bigl(A_{K'} - KC\bigr)e + W_{K'}\delta.
        \end{equation}
        Gain $K'$ replaces $W$ by $W_{K'}$ in the coupling term, while $K$
        damps the linear part and $E$ shifts the sector argument.
  \item Solve for $(K,K',E,P,\Lambda)$ to satisfy the LMI \eqref{eq:LMI}
        (Theorem~\ref{thm:main}).
\end{enumerate}

The architecture adds $2n_y$ scalar states to the $n$ states of
the Lur'e-type observer; when $n_y \ll n$ this overhead is negligible.
The second-order sliding-mode gains $L_j$ must be tuned per channel and the discontinuous
right-hand side requires a dedicated integration scheme~\cite{Moreno2008}.

\begin{proposition}[Homogeneous observer]
  \label{prop:sliding}
  Under Assumption~\ref{ass:compact}, consider the homogeneous observer
  (upper block of~\eqref{eq:observer}), restated per channel
  $j=1,\ldots,n_y$ as
  \begin{align}
    \dot{\hat{z}}_{1,j} &= \hat{z}_{2,j} + [CBu]_j
      + L_j k_1 \spow{y_j - \hat{z}_{1,j}}^{1/2},
      \label{eq:homo_1} \\
    \dot{\hat{z}}_{2,j} &\in L_j^2 k_2\,\sign(y_j
      - \hat{z}_{1,j}),
      \label{eq:homo_2}
  \end{align}
  under the perturbed system~\eqref{eq:system}.
  There exist constants $k_1,k_2$ such that, for any $\bar{p},\bar{\nu}>0$,
  there exist $L_j^*\geq 1$, a class-$\mathcal{KL}$ function $\beta$, and
  a constant $\gamma>0$ (depending on $\bar{p},\bar{\nu}$ and on the
  system data: $\|W\|$, $\|C\|$, the regularity of $f$, and
  $\mathcal{K}$) such that, for all $L_j\geq L_j^*$, any
  solution of the combined system~\eqref{eq:system}--\eqref{eq:homo_2}
  satisfies, for all $t\geq 0$,
  \begin{multline}
    \label{eq:ISS_homogeneous}
    |\hat{z}_{2,j}(t) - [CAx(t)+CWf(x(t))]_j|
    \leq \\ \max\Bigl(\beta\bigl(|\hat{z}_j(0)-z_j(0)|,\,t\bigr),\;
    \gamma\bigl(L_j\,\bar{\nu}^{1/2}
    + \bar{p}\bigr)\Bigr).
  \end{multline}
  In the absence of noise and disturbances ($\bar{\nu}=\bar{p}=0$),
  $\hat{z}_{2,j}(t)=[CAx(t)+CWf(x(t))]_j$ for all $t\geq T_0$, for
  some $T_0>0$.
\end{proposition}

\begin{proof}
  Define $z_{1,j}=[Cx]_j$ and $z_{2,j}=[CAx+CWf(x)]_j$.
  Differentiating along~\eqref{eq:system} with disturbance $p$,
  $\dot{z}_{1,j}=z_{2,j}+[CBu]_j+[Cp]_j$,
  $y_j=z_{1,j}+\nu_j$, which is the triangular
  form~\cite{bernard2017observers} with $m=2$, $\Phi_1=[CBu]_j$,
  $w_1=[Cp]_j$, and $v=\nu_j$.
  The second channel is
  \begin{multline*}
    \dot{z}_{2,j}
    = \bigl[C(I + W\partial_x f(x))(Ax + Wf(x) \\
         + p(t) + Bu(t))\bigr]_j
    := \Phi_2(z_{1,j},z_{2,j},x(t),t).
  \end{multline*}
  Under Assumption~\ref{ass:compact}, $x\in\mathcal{K}$, $f$ is
  Lipschitz, and $\Phi_2$ involves $A$, $W$, $C$, $f(x)$, and the
  Jacobian of $f$ along the flow; hence $\Phi_2$ is uniformly bounded
  by a constant depending on $\|W\|$, $\|C\|$, the Lipschitz constant
  of $f$, and $\mathcal{K}$.
  Crucially, the expression of $\Phi_2$ need not be known.
  The observer~\eqref{eq:homo_1}--\eqref{eq:homo_2} is the homogeneous
  observer of~\cite{bernard2017observers} with $d_0=-1$ and gains
  $(L_jk_1,\,L_j^2k_2)$.  For $m=2$ and $d_0=-1$ the weight vector
  is $(r_1,r_2,r_3)=(2,1,0)$, which gives the correction exponents
  $r_2/r_1=1/2$ and $r_3/r_1=0$ in~\eqref{eq:homo_1}--\eqref{eq:homo_2}.
  The bound~\eqref{eq:ISS_homogeneous} is
  \cite[Proposition~4, Eq.~(16)]{bernard2017observers} evaluated at the
  second component $i=2$: the noise term is
  $L_j^{i-1}|v|^{r_2/r_1}=L_j\bar{\nu}^{1/2}$, while the first-channel
  disturbance $w_1=[Cp]_j$ enters with exponent $r_2/r_2=1$ and gain
  $L_j^{-\mu}$ with $\mu=(1-i+1)\,r_1/r_2=0$, i.e.\ without attenuation.
  The latter is unavoidable: since $\dot{z}_{1,j}=z_{2,j}+\Phi_1+w_1$, a
  disturbance on the first channel is indistinguishable from a shift of
  $z_{2,j}$.  Only channels $1,\ldots,i-1$ appear in Eq.~(16), so the
  disturbance $\Phi_2$ on the second channel does not enter the
  steady-state estimate.
  The noiseless convergence is the special case $\bar{\nu}=\bar{p}=0$.
\end{proof}

\begin{remark}[Convergence rate versus noise amplification]
  The $\mathcal{KL}$ function $\beta$ in~\eqref{eq:ISS_homogeneous}
  can be made to be equal to $0$ arbitrarily fast by increasing $L_j$. In other words, for any prescribed
  $T_0>0$, there exists $L_j$ large enough such that 
  for $t\geq T_0$, the reconstruction error reduces to
  $|\hat{z}_{2,j}-z_{2,j}| \leq \gamma(L_j\bar{\nu}^{1/2} + \bar{p})$.
  Increasing $L_j$ thus reduces the convergence time but amplifies the
  $L_j$ factor in the noise floor. This is the trade-off
  between speed and noise sensitivity
  ~\cite{Levant2003,Moreno2008,Moreno2012,bernard2017observers}.
\end{remark}

\begin{remark}[H\"older nonlinearity with $f(0)=0$, $\alpha\geq1/2$]
  \label{rem:holder}
  The proof of Proposition~\ref{prop:sliding} requires $\Phi_2 = \dot{z}_2$ to
  be bounded along trajectories, which is not granted when $f'$ is unbounded.
  It nevertheless holds in the following case.  When $f(0)=0$ and $f$ is
  component-wise H\"older with exponent $\alpha\geq 1/2$, the product
  $f'(x_i)\dot{x}_i$ remains bounded along trajectories of the Hurwitz system
  $\dot{x} = -a x + Wf(x)$: near the origin $\dot{x}=O(\|x\|^{\alpha})$ while
  $f'(x)=O(\|x\|^{\alpha-1})$, so $f'(x_i)\dot{x}_i = O(\|x\|^{2\alpha-1})$,
  which is $O(1)$ for $\alpha\geq 1/2$.  Proposition~\ref{prop:sliding}
  therefore applies.
\end{remark}

\begin{theorem}[Combined observer]
  \label{thm:main}
  Under Assumption~\ref{ass:compact}, suppose there exist $P \succ 0$,
  diagonal $\Lambda \succ 0$, $q > 0$, and matrices $K, K', E$ such that
  \begin{equation}
    \label{eq:LMI}
    \mathcal{M} := \begin{pmatrix}
      M_{11} & M_{12} \\ M_{12}^\top & M_{22}
    \end{pmatrix} \preceq 0 ,
  \end{equation}
  where
  \begin{subequations}
    \label{eq:blocks}
    \begin{align}
      M_{11} &= \He\{P(A_{K'} - KC)\} + qI, \label{eq:M11}\\
      M_{12} &= P W_{K'} + (I-EC)^\top\Gamma\Lambda, \label{eq:M12}\\
      M_{22} &= -2\Lambda, \label{eq:M22}
    \end{align}
  \end{subequations}
  with $A_{K'}$, $W_{K'}$ as in~\eqref{eq:shorthand}.
  Then, in the noise-free case ($\nu \equiv 0$, $p \equiv 0$), for any
  $L_j \geq L_j^*$ ($j=1,\ldots,n_y$), where
  $L_j^*$ are the thresholds from Proposition~\ref{prop:sliding},
  the observer \eqref{eq:observer} achieves
  $\|e(t)\| \leq \kappa_P\, e^{-\rho(t-T_0)}\|e(T_0)\|$ for all $t \geq T_0$,
  with $T_0$ the finite reconstruction time of
  Proposition~\ref{prop:sliding},
  $\kappa_P = \sqrt{\lambda_{\max}(P)/\lambda_{\min}(P)}$ and
  $\rho = q/(2\lambda_{\max}(P))$.
\end{theorem}

\begin{proof}
  For $t\geq T_0$, Proposition~\ref{prop:sliding} gives $\hat{z}_2(t) = y_2(t)$
  exactly.  Thus $e = \hat{x}-x$ satisfies \eqref{eq:error} with
  $\delta = f(\hat{\zeta})-f(x)$ and $\hat{\zeta}-x = (I-EC)e$.  Applying
  the sector condition \eqref{eq:sector} with $\hat{x}=\hat{\zeta}$:
  \begin{equation}
    \label{eq:sector_applied}
    \sum_{i=1}^n \lambda_i\,\delta_i\bigl(\delta_i - \gamma_i[(I-EC)e]_i\bigr) \leq 0,
  \end{equation}
  that is, in matrix form,
  $\delta^\top\Lambda(\delta - \Gamma (I-EC)e) \leq 0$.
  Consider $\mathcal{V}(e) = e^\top Pe$.  Its derivative along \eqref{eq:error} is
  \begin{equation*}
    \dot{\mathcal{V}} = e^\top\He\{P(A_{K'} - KC)\}e
                        + 2\,e^\top P W_{K'}\delta.
  \end{equation*}
  Expanding the blocks~\eqref{eq:blocks} for
  $\xi = (e^\top,\delta^\top)^\top$ gives the identity
  \begin{equation*}
    \xi^\top\mathcal{M}\xi
      = \dot{\mathcal{V}} + q\|e\|^2
        - 2\,\delta^\top\Lambda\bigl(\delta - \Gamma (I-EC)e\bigr),
  \end{equation*}
  which is the standard S-procedure step for Lur'e-type
  observers~\cite{Arcak2001,Fan2003}.  Since the last term is nonnegative
  by~\eqref{eq:sector_applied} and $\xi^\top\mathcal{M}\xi\leq0$
  by~\eqref{eq:LMI},
  \begin{equation*}
    \dot{\mathcal{V}} + q\|e\|^2 \leq \xi^\top\mathcal{M}\xi \leq 0 .
  \end{equation*}
  Hence $\dot{\mathcal{V}} \leq -q\|e\|^2 \leq
  -\tfrac{q}{\lambda_{\max}(P)}\mathcal{V}$, giving
  $\mathcal{V}(e(t)) \leq \mathcal{V}(e(T_0))\,
  e^{-\frac{q}{\lambda_{\max}(P)}(t-T_0)}$,
  and the stated bound follows from $\lambda_{\min}(P)\|e\|^2 \leq \mathcal{V}
  \leq \lambda_{\max}(P)\|e\|^2$.  The error $e(T_0)$ is finite by
  Remark~\ref{rem:transient}.
\end{proof}

\begin{corollary}[Output-linear drift]
  \label{cor:outlin}
  Suppose $CA = MC$ for a known matrix $M \in \R^{n_y\times n_y}$.
  Then $CAx = My$ is directly available from the output; the
  sliding-mode observer only needs to reconstruct
  $CWf(x)$.  The observer simplifies to
  \begin{equation}
    \label{eq:observer_wc}
    \begin{cases}
      \begin{aligned}
        \dot{\hat{z}}_{1,j} &= [M\hat{z}_1]_j + \hat{z}_{2,j} + [CBu]_j \\
          &\quad + L_j k_1\spow{y_j - \hat{z}_{1,j}}^{1/2},
      \end{aligned} \\[3pt]
      \dot{\hat{z}}_{2,j} \in  L_j^2k_2\,\sign(y_j - \hat{z}_{1,j}),
        \qquad j=1,\ldots,n_y, \\[5pt]
      \begin{aligned}
        \dot{\hat{x}} &= A\hat{x} + Wf(\hat{\zeta})
          + K(y - C\hat{x}) \\
        &\quad - K'\bigl(CWf(\hat{\zeta}) - \hat{z}_2\bigr) + Bu,
      \end{aligned}
    \end{cases}
  \end{equation}
  and the error dynamics reduce to
  $\dot{e} = (A-KC)e + W_{K'}\delta$.
  Theorem~\ref{thm:main} holds under the same assumptions with the
  LMI~\eqref{eq:LMI} replaced by the matrix $\mathcal{M}_1$, defined by
  \begin{equation}
    \label{eq:LMI_1}
    \mathcal{M}_1: \quad M_{11} = \He\{P(A - KC)\} + qI ,
  \end{equation}
  the blocks $M_{12}$, $M_{22}$ being unchanged.  The only difference
  from~\eqref{eq:M11} is the absence of the $-PK'CA$ term, because $CAx$ is
  known.  A relevant model with this property is the Wilson--Cowan
  neural-mass model~\cite{Wilson1972,wilson1973mathematical}: its drift is
  $A = -a I_n$, so $CA = MC$ holds with $M = -a I_{n_y}$.
\end{corollary}

\begin{proposition}[Direct nonlinear output]
  \label{prop:direct}
  Suppose $f$ acts component-wise and each row of $C$ is a standard basis
  vector of $\R^n$ (each output channel measures exactly one state
  component).  Then $Cf(x) = f(Cx) = f(y)$ is directly available from the
  measured output, and the second correction channel of~\eqref{eq:observer}
  can be driven by $f(y)$ instead of $\hat{z}_2$:
  \begin{equation}
    \label{eq:observer_direct}
    \begin{aligned}
      \dot{\hat{x}} &= A\hat{x} + Wf(\hat{\zeta}) + K(y - C\hat{x}) \\
        &\quad - K'\bigl(Cf(\hat{\zeta}) - f(y)\bigr) + Bu .
    \end{aligned}
  \end{equation}
  The family of homogeneous observers and its tuning are eliminated
  entirely.  The error dynamics are
  $\dot{e} = (A-KC)e + (W-K'C)\delta$, and Theorem~\ref{thm:main} holds
  with the LMI~\eqref{eq:LMI} replaced by the matrix $\mathcal{M}_2$,
  defined by
  \begin{equation}
    \label{eq:LMI_2}
    \mathcal{M}_2: \quad
    \begin{aligned}
      M_{11} &= \He\{P(A - KC)\} + qI, \\
      M_{12} &= P(W - K'C) + (I-EC)^\top\Gamma\Lambda ,
    \end{aligned}
  \end{equation}
  and $M_{22}$ unchanged.
\end{proposition}

Note that here $K'$ acts on the coupling as $W-K'C$ rather than as
$W_{K'}=(I-K'C)W$: the reconstructed signal is $Cf(x)$, not $CWf(x)$, so
Corollary~\ref{cor:outlin} does not apply.  Note also that the injected signal
$f(y)=f(Cx+\nu)$ carries the measurement noise through $f$: for $f$
H\"older with exponent $\alpha$ and constant $\ell$, this contributes a term
$\|K'\|\,\ell\,\bar{\nu}^{\alpha}$ to the steady-state error --- the same
$\bar{\nu}^{1/2}$ scaling as the reconstruction error of
Proposition~\ref{prop:sliding} when $\alpha=1/2$.

\begin{remark}[Increasing-only LMI]
  \label{rem:increasing}
  When the nonlinearity is only known to be component-wise increasing
  ($\delta_i e_i \geq 0$) but no slope bound $\gamma_i$ is available,
  the sector condition~\eqref{eq:sector} does not apply and the LMI
  \eqref{eq:LMI} cannot be used.  Applying the S-procedure with the
  increasing condition $\delta^\top \Lambda (I-EC)e \geq 0$
  ($\Lambda=\diag(\lambda_i)\succ 0$) to the Lyapunov derivative
  $\dot{\mathcal{V}}$ yields the same block structure~\eqref{eq:LMI} with
  $M_{12} = PW_{K'} + (I-EC)^\top\Lambda$ and $M_{22} = 0$.  A matrix with a
  vanishing $(2,2)$ block is negative semidefinite only if its off-diagonal
  block vanishes as well---a condition that is rarely feasible
  ($PW + (I-EC)^\top\Lambda = 0$ for the classical observer).  Adding a small
  regularisation $-\varepsilon I$ ($\varepsilon>0$) in the $(2,2)$ block
  yields the relaxed LMI $\mathcal{M}_3$, defined by
  \begin{equation}
    \label{eq:LMI_3}
    \mathcal{M}_3: \quad
    \begin{aligned}
      M_{12} &= P W_{K'} + (I-EC)^\top\Lambda, \\
      M_{22} &= -\varepsilon I .
    \end{aligned}
  \end{equation}
  Compared with~\eqref{eq:M12}, the term $(I-EC)^\top\Gamma\Lambda$ is
  replaced by $(I-EC)^\top\Lambda$ and the damping term $-2\Lambda$ by
  $-\varepsilon I$;
  $M_{11}$ is unchanged.  The combined observer retains its coupling
  attenuation because $K'$ replaces $W$ by $W_{K'}$ in the
  off-diagonal independently of $\Gamma$ (and independently of~$E$).

  The finite-time convergence of Proposition~\ref{prop:sliding} requires
  $\dot{z}_2$ to be bounded, which is not granted when $f'$ is unbounded;
  Remark~\ref{rem:holder} gives a class of H\"older nonlinearities for which
  it still holds, and the simulations of Sections~\ref{sec:case1}
  and~\ref{sec:case2} confirm that the reconstruction succeeds in that
  regime, with a significant performance gain for the proposed design.
\end{remark}

\begin{remark}[Transient on $[0,T_0)$]
  \label{rem:transient}
  On $[0,T_0)$, the reconstruction error $\varepsilon_2(t)$ is bounded by
  construction~\cite{Moreno2008,bernard2017observers}, and the observer error dynamics
  $\dot{e} = (A_{K'} - KC)e + W_{K'}\delta + K'\varepsilon_2(t)$,
  as in~\eqref{eq:error}, are locally Lipschitz in $e$
  (with $\varepsilon_2$ acting as a bounded exogenous input);
  hence $e(T_0)$ is finite.
  The exponential bound of Theorem~\ref{thm:main} therefore starts from a
  bounded initial condition.
\end{remark}

\section{Comparison with the standard Lur'e-type LMI}
\label{sec:comparison}

The standard Lur'e-type LMI \cite{Arcak2001,Fan2003,Zemouche2013,giaccagli2023lmi}
is recovered by setting $K'=0$ in~\eqref{eq:LMI} (so that $A_{K'}=A$,
$W_{K'}=W$); we denote the resulting matrix by $\mathcal{M}_{\mathrm{std}}$:
\begin{equation}
  \label{eq:LMI_std}
  \mathcal{M}_{\mathrm{std}}: \quad
  \begin{aligned}
    M_{11} &= \He\{P(A - KC)\} + qI, \\
    M_{12} &= PW + (I-EC)^\top\Gamma\Lambda ,
  \end{aligned}
\end{equation}
with $M_{22}$ unchanged.
We denote the observer gains delivered by the standard
LMI~\eqref{eq:LMI_std} and the combined LMI~\eqref{eq:LMI} by
$K_{\mathrm{std}}$ and $K_{\mathrm{comb}}$, respectively.
The coupling matrix $W$ enters the block $M_{12}$ of
$\mathcal{M}_{\mathrm{std}}$ without any design freedom; feasibility requires $P$ and $\Lambda$ to jointly
absorb $\|W\|$.  In contrast, \eqref{eq:M12} replaces $W$ by $W_{K'}$:
choosing $K'$ to reduce the effective norm
$\|W_{K'}\|$ can lower the observer gains needed to certify a given
convergence rate, which is especially useful when the classical design is
formally feasible but its gains become very large in practice.

\begin{proposition}[Inclusion of feasibility sets]
  \label{prop:inclusion}
  Any $(P,\Lambda,K,E)$ feasible for \eqref{eq:LMI_std} is feasible for
  \eqref{eq:LMI} with $K' = 0$.  Hence the feasibility
  set of \eqref{eq:LMI} contains that of \eqref{eq:LMI_std}.
  This is immediate, since \eqref{eq:LMI} evaluated at $K'=0$
  is \eqref{eq:LMI_std}.
\end{proposition}

\begin{proposition}[Fixed $K'$ choice for coupling reduction]
  \label{prop:general}
  Assume $C$ has full row rank.  For arbitrary $A$, $W$, $\Gamma$, the choice $K' = C^\top(CC^\top)^{-1}$
  yields the effective coupling
  \begin{equation}
    \label{eq:opt_Kprime}
    \begin{aligned}
    W_{K'} &= \Pi_C\,W, \\
    \Pi_C &= I - C^\top(CC^\top)^{-1}C,
    \end{aligned}
  \end{equation}
  where $\Pi_C$ is the orthogonal projector onto $\ker C$
  ($\Pi_C^\top = \Pi_C = \Pi_C^2$).
  Consequently the coupling matrix appearing in $M_{12}$
  satisfies $\|W_{K'}\| = \|\Pi_C\,W\| \leq \|W\|$,
  with equality iff $CW = 0$ and strict inequality whenever $CW \neq 0$.

  If, in addition, the columns of $W$ lie in $\operatorname{range}(C^\top)$,
  i.e.\ $W = C^\top X$ for some $X$, then $W_{K'} = 0$: the nonlinear
  coupling is completely eliminated from the error dynamics, and the
  combined LMI~\eqref{eq:LMI} reduces to a linear stability condition
  independent of $W$.  Under the further condition $CA = MC$
  (Corollary~\ref{cor:outlin}), the LMI $\mathcal{M}_1$~\eqref{eq:LMI_1}
  applies.  If $A$ is Hurwitz, that LMI is then feasible with $K = 0$:
  the error dynamics are $\dot{e} = A e$, and the sector condition is
  absorbed by choosing $\Lambda$ small enough.
\end{proposition}

\begin{proof}
  With $K' = C^\top(CC^\top)^{-1}$,
  $W_{K'} = (I - C^\top(CC^\top)^{-1}C)W = \Pi_C\,W$.
  Since $\Pi_C$ is an orthogonal projector, $\|\Pi_C\,W\| \leq \|W\|$,
  with equality iff $CW = 0$.
  When $W = C^\top X$, $W_{K'} = \Pi_C\,C^\top X = 0$, so
  $M_{12}$ reduces to $(I-EC)^\top\Gamma\Lambda$.
  With $K=0$ and $CA=MC$, the $(1,1)$ block is $\He\{PA\} + qI$.  Since $A$ is
  Hurwitz, $P \succ 0$ can be chosen to satisfy $\He\{PA\} = -Q \prec 0$
  for any $Q \succ 0$.  As $M_{22} = -2\Lambda \prec 0$, the Schur complement
  of the full matrix with respect to $M_{22}$ is
  \begin{equation*}
    \He\{PA\} + qI + \tfrac{1}{2}\,(I-EC)^\top\Gamma\Lambda\Gamma (I-EC) ,
  \end{equation*}
  which is negative definite for $q$ and $\Lambda = \varepsilon I$
  ($\varepsilon>0$) small enough.  Hence the LMI $\mathcal{M}_1$
  \eqref{eq:LMI_1} is feasible.
\end{proof}

\begin{remark}[Structure of the combined design]
  The classical LMI~\eqref{eq:LMI_std} forces a single gain $K$ to
  simultaneously stabilise the linear error dynamics and absorb the coupling
  $PW$, both through the same Lyapunov matrix $P$.  The combined design
  separates these tasks into three specialised parameters:
  \begin{enumerate}[(i)]
    \item $K'$ attenuates coupling geometrically,
          replacing $PW$ by $P\,\Pi_C\,W$ in $M_{12}$
          (Proposition~\ref{prop:general});
    \item $K$ stabilises the residual linear error dynamics with attenuated
          coupling;
    \item $L_j$ ($j=1,\ldots,n_y$) reconstruct the coupling components
          visible through the output, each governed by an existential
          condition: there exists a threshold $L_j^*$ such that exact
          reconstruction holds for all $L_j \geq
          L_j^*$~\cite{Levant2003,Moreno2008,bernard2017observers}.
  \end{enumerate}
\end{remark}

\begin{remark}[Trade-off in the choice of $K'$]
  While $K' = C^\top(CC^\top)^{-1}$ eliminates the coupling components in the
  row space of $C$, the combined LMI~\eqref{eq:LMI} also involves $K'$ in
  $M_{11}$ through $PK'CA$.  Choosing $K'$ solely to cancel the coupling
  may destabilise the linear part $A_{K'}$ if $A$ has unstable modes in the
  output directions.  In the case $A = -a I$, this
  trade-off disappears because $A$ and $K'CA$ commute and both contribute
  damping.  More generally, the LMI~\eqref{eq:LMI} jointly optimises over
  $(K,K',P,\Lambda)$ and automatically balances coupling attenuation against
  linear stability.  The closed-form $K'$ of Proposition~\ref{prop:general}
  serves as an admissible choice with a guaranteed coupling reduction, not
  necessarily as the LMI-optimal $K'$.
\end{remark}

\section{Stability under noise and model disturbances}
\label{sec:iss}

In practice the output $y$ is corrupted by measurement noise and the model may be subject to a bounded disturbance.  The following result quantifies how both perturbations propagate to the observer error, and how the combined observer's smaller effective gain reduces the noise floor relative to the classical design.

Under the perturbed system~\eqref{eq:system}, the per-channel evolution is
$\dot{z}_{1,j} = z_{2,j} + [CBu]_j + [Cp]_j$,
$y_j = z_{1,j} + \nu_j$.
Proposition~\ref{prop:sliding} applies; choosing $L_j$ sufficiently large such that the $\mathcal{KL}$
transient is below the steady-state floor for $t\geq T_0$,
the bound~\eqref{eq:ISS_homogeneous} reduces to the
component-wise estimate
\begin{equation}
  \label{eq:sta_error}
  \begin{aligned}
    \hat{z}_{2,j}(t) &= z_{2,j}(t) + \varepsilon_{2,j}(t), \\
    |\varepsilon_{2,j}(t)| &\leq c_0\,\bar{\nu}^{1/2} + c_1\,\bar{p}, \\
    \quad c_0 &= O(L_j),\; c_1 = O(1),
  \end{aligned}
\end{equation}
for $j = 1,\ldots,n_y$, where $c_0\bar{\nu}^{1/2}$ is the measurement noise contribution
and $c_1\bar{p}$ is the propagation of the disturbance $[Cp]_j$
through the observer~\cite{Levant2003,Moreno2008,bernard2017observers};
as noted in Proposition~\ref{prop:sliding}, the latter is not attenuated
by $L_j$.
The component-wise bounds imply
$\|\varepsilon_2\| \leq \sqrt{n_y}\,(c_0\bar{\nu}^{1/2} + c_1\bar{p})
:= \bar{\varepsilon}_2$.

\begin{proposition}[Stability under noise and model disturbances]
  \label{prop:iss}
  Suppose the LMI~\eqref{eq:LMI} is feasible with $E = 0$ and rate $q > 0$,
  and $\hat{z}_{2}$ satisfies~\eqref{eq:sta_error} for $t \geq T_0$.
  Define $\bar{\varepsilon}_2 = \sqrt{n_y}\,(c_0\bar{\nu}^{1/2} + c_1\bar{p})$.
  Then the error $e(t) = \hat{x}(t) - x(t)$ satisfies, for $t \geq T_0$,
  \begin{multline}
    \label{eq:iss_bound}
    \|e(t)\| \leq \kappa_P\Bigl(e^{-\rho(t-T_0)}\|e(T_0)\|
             + \\ \frac{1}{\rho}\bigl(\|K\|\bar{\nu}
               + \|K'\|\bar{\varepsilon}_2 + \bar{p}\bigr)\Bigr),
  \end{multline}
  where $\kappa_P = \sqrt{\lambda_{\max}(P)/\lambda_{\min}(P)}$ is the condition number of $P$
  and $\rho = q/(2\lambda_{\max}(P))$.
\end{proposition}

\begin{proof}
  For $t \geq T_0$, the measurement is $y_{\mathrm{meas}} = Cx + \nu$ with $\|\nu\| \leq \bar{\nu}$,
  and~\eqref{eq:sta_error} gives $\hat{z}_2 = z_2 + \varepsilon_2$ with
  $\|\varepsilon_2\| \leq \bar{\varepsilon}_2$.
  With $E = 0$, we have $\hat{\zeta} = \hat{x}$, hence $\delta = f(\hat{x}) - f(x)$.
  Substituting $y_{\mathrm{meas}}$ and $\hat{z}_2$ into the observer~\eqref{eq:observer},
  the process disturbance $p$ enters $\dot{e}$ as $-p$ (from the plant dynamics),
  and the reconstruction error $\varepsilon_2$ enters through the correction
  term $-K'(CA\hat{x}+CWf(\hat{x})-\hat{z}_2) = -K'(CAe + CW\delta - \varepsilon_2)$,
  contributing $+K'\varepsilon_2$.
  The full error dynamics are
  \begin{multline*}
    \dot{e} = (A_{K'} - KC)e \\
              + W_{K'}\delta + K\nu(t) + K'\varepsilon_2(t) - p(t).
  \end{multline*}
  Define the $P$-norm $\|e\|_P = \sqrt{e^\top P e}$.
  Differentiating $\|e\|_P^2 = e^\top P e$ along~$\dot{e}$,
  \begin{multline*}
    \frac{d}{dt}\|e\|_P^2
      = e^\top\He\{P(A_{K'} - KC)\}e + 2e^\top P W_{K'}\delta \\
        + 2e^\top P\bigl(K\nu + K'\varepsilon_2 - p\bigr).
  \end{multline*}
  The LMI~\eqref{eq:LMI} (with $E=0$) absorbs the $\delta$ terms exactly
  as in the proof of Theorem~\ref{thm:main}, giving
  $e^\top\He\{P(A_{K'} - KC)\}e + 2e^\top P W_{K'}\delta \leq -q\|e\|^2$.
  Hence
  \begin{equation}
    \frac{d}{dt}\|e\|_P^2
      \leq -q\|e\|^2 + 2\,e^\top P\bigl(K\nu + K'\varepsilon_2 - p\bigr).
  \end{equation}
  Writing $e^\top P K \nu = (P^{1/2}e)^\top P^{1/2}K \nu$ and applying
  Cauchy--Schwarz,
  \begin{multline*}
    e^\top P\bigl(K\nu + K'\varepsilon_2 - p\bigr)
      \leq \|e\|_P\bigl(\|P^{1/2}K\|\bar{\nu}
        + \\ \|P^{1/2}K'\|\bar{\varepsilon}_2
        + \|P^{1/2}\|\bar{p}\bigr).
  \end{multline*}
  Also $\|e\|_P^2 \leq \lambda_{\max}(P)\|e\|^2$, so $-q\|e\|^2 \leq -\frac{q}{\lambda_{\max}(P)}\|e\|_P^2$.
  Putting these together,
  \begin{multline*}
    \frac{d}{dt}\|e\|_P^2
      \leq -\frac{q}{\lambda_{\max}(P)}\|e\|_P^2 \\
        + 2\|e\|_P\bigl(\|P^{1/2}K\|\bar{\nu}
          + \|P^{1/2}K'\|\bar{\varepsilon}_2
          + \|P^{1/2}\|\bar{p}\bigr).
  \end{multline*}
  Dividing by $2\|e\|_P$, wherever $\|e\|_P>0$ and in the sense of the upper
  Dini derivative elsewhere,
  \begin{multline*}
    \dot{\|e\|}_P
      \leq -\underbrace{\frac{q}{2\lambda_{\max}(P)}}_{\textstyle\rho}\|e\|_P \\
        + \|P^{1/2}K\|\bar{\nu}
        + \|P^{1/2}K'\|\bar{\varepsilon}_2
        + \|P^{1/2}\|\bar{p}.
  \end{multline*}
  This is a linear differential inequality in $\|e\|_P$.
  By the comparison lemma~\cite{Khalil2002},
  \begin{multline*}
    \|e\|_P(t) \leq \|e\|_P(T_0)e^{-\rho(t-T_0)} \\
      + \frac{1}{\rho}\bigl(\|P^{1/2}K\|\bar{\nu}
        + \|P^{1/2}K'\|\bar{\varepsilon}_2
        + \|P^{1/2}\|\bar{p}\bigr).
  \end{multline*}
  Converting back to the Euclidean norm via
  $\sqrt{\lambda_{\min}(P)}\|e\| \leq \|e\|_P \leq \sqrt{\lambda_{\max}(P)}\|e\|$
  and bounding $\|P^{1/2}\| = \sqrt{\lambda_{\max}(P)}$,
  $\|P^{1/2}X\| \leq \sqrt{\lambda_{\max}(P)}\|X\|$ for each gain matrix
  gives $\kappa_P = \sqrt{\lambda_{\max}(P)/\lambda_{\min}(P)}$ in~\eqref{eq:iss_bound}.
\end{proof}

\begin{remark}[Noise floor comparison and crossover]
  At steady state, \eqref{eq:iss_bound} gives the exact bounds
  \begin{multline}
    \limsup_{t\to\infty}\|e(t)\| \leq \bar{e}_{\mathrm{comb}}
    := \frac{\kappa_P}{\rho}\Bigl(\|K_{\mathrm{comb}}\|\bar{\nu} \\
    + \|K'\|\,(c_0\sqrt{n_y}\,\bar{\nu}^{1/2}
    + c_1\sqrt{n_y}\,\bar{p}) + \bar{p}\Bigr)
  \end{multline}
  for the combined design, and
  $\limsup_{t\to\infty}\|e(t)\| \leq \bar{e}_{\mathrm{std}}
  := \frac{\kappa_P}{\rho}(\|K_{\mathrm{std}}\|\bar{\nu} + \bar{p})$
  for the standard design.
  The combined observer benefits from $\|K_{\mathrm{comb}}\| \ll \|K_{\mathrm{std}}\|$
  but pays a $\bar{\nu}^{1/2}$-scaling penalty through the reconstruction noise
  $c_0\sqrt{n_y}\,\bar{\nu}^{1/2}$, and an additional
  $\|K'\|c_1\sqrt{n_y}\,\bar{p}$ on the disturbance channel.
  For $\bar{p}=0$, the classical observer outperforms below a crossover noise
  level scaling as
  $\bar{\nu}_{\mathrm{cross}} \sim \bigl(\|K'\|\sqrt{n_y}\,c_0/\|K_{\mathrm{std}}\|\bigr)^2$.
  Since the constant hidden in $c_0 = O(L_j)$ is not identified, this
  predicts the dependence on $\|K'\|/\|K_{\mathrm{std}}\|$ rather than the crossover
  value itself; the crossover is located numerically in
  Section~\ref{sec:simulations}.
\end{remark}

\section{Numerical simulations}
\label{sec:simulations}

\subsection{Model and simulation setup}

We consider a generalized Wilson--Cowan model \cite{Wilson1972,wilson1973mathematical}
with $n_b = 3$ excitatory and $3$ inhibitory nodes ($n = 6$).
Grouping the excitatory states $x_e \in \R^{n_b}$ and inhibitory
states $x_i \in \R^{n_b}$ into $x = (x_e^\top, x_i^\top)^\top$, the
dynamics read
\begin{equation}
  \label{eq:wc}
  \begin{cases}
    \dot{x}_e = A_{ee} x_e + A_{ei} x_i + W_{ee} f(x_e) + W_{ei} f(x_i), \\[3pt]
    \dot{x}_i = A_{ie} x_e + A_{ii} x_i + W_{ie} f(x_e) + W_{ii} f(x_i),  \\
    y = Cx,
  \end{cases}
\end{equation}
with block connectivity matrix
\begin{equation}
  W = \begin{pmatrix} W_{ee} & W_{ei} \\ W_{ie} & W_{ii} \end{pmatrix},
  \qquad W_s = s\,W,\quad s > 0,
\end{equation}
where $s$ scales the overall coupling strength.
Unless stated otherwise, the measurement matrix $C \in \R^{n_y\times n}$
is such that each output channel is a linear combination of two state
components, one from the $x_e$ block and one from the $x_i$ block
($n_y = n_b = 3$).
All simulations use $n=6$, $n_y=3$, a baseline coupling
$\|W_0\| = 4.4$, and homogeneous observer gains $L_j = 3$ ($j=1,2,3$)
with $k_1 = 1.5$, $k_2 = 1.1$ in~\eqref{eq:homo_1}--\eqref{eq:homo_2}.
The plant and observers are integrated with RK4 at $\Delta t = 10^{-3}$\,s
for $10$\,s.  Zero-mean Gaussian noise of standard deviation $\sigma$
corrupts $y$; process noise of the same level is added to the plant.
For each design, the gains are obtained by maximising the certified rate $q$,
capped at $q_{\max}=10$, subject to the corresponding LMI in the variables
$(P, PK, PK', \Lambda)$ with $P \succeq \varepsilon_P I$ and
$\Lambda \succeq \varepsilon_\Lambda I$
($\varepsilon_P = \varepsilon_\Lambda = 10^{-4}$, and
$\varepsilon = 10^{-3}$ in~\eqref{eq:LMI_3});
the reported norms are those of $K = P^{-1}(PK)$ and $K' = P^{-1}(PK')$
at the optimum.
The root-mean-square (RMS) estimation error
\begin{equation*}
  \|e\|_{\mathrm{rms}}
  = \Bigl(\frac{1}{4}\int_{6}^{10}\|e(t)\|^2\,\mathrm{d}t\Bigr)^{1/2}
\end{equation*}
is computed over the last $40\%$ of each trajectory.

We present three cases of increasing complexity, each highlighting a
different facet of the combined design.

\subsection{Case 1: Direct output, no sliding-mode reconstruction}
\label{sec:case1}

We first illustrate the simplest scenario: Proposition~\ref{prop:direct},
where each output channel measures exactly one state component
($C = [I_3\;\;0_{3\times 3}]$).  Then $Cf(x) = f(Cx) = f(y)$ is directly
available from the output and the family of homogeneous observers is
unnecessary.  The observer~\eqref{eq:observer} reduces
to~\eqref{eq:observer_direct} with $E=0$,
\begin{equation}
  \dot{\hat{x}} = A\hat{x} + Wf(\hat{x}) + K(y - C\hat{x})
                  - K'\bigl(Cf(\hat{x}) - f(y)\bigr),
\end{equation}
with $A$ a general Hurwitz matrix ($\max\Re\lambda(A) = -0.77$).
The error dynamics are $\dot{e} = (A-KC)e + (W-K'C)\delta$, certified by the
blocks~\eqref{eq:LMI_2}.

To demonstrate that the coupling attenuation via $K'$ does not rely
on a slope bound, we use the H\"older nonlinearity
$f(\xi) = |\xi|^{1/2}\sign(\xi)/(1+|\xi|^{1/2})$, which is bounded,
strictly increasing, and satisfies $f'(0)=+\infty$, i.e.\ no global sector
bound exists.  The observer gains are computed with the increasing-only
LMI (no $\Gamma$) of Remark~\ref{rem:increasing}, whose off-diagonal block is
here $M_{12} = P(W-K'C) + \Lambda$.

Table~\ref{tab:case1} reports the gain norms.  The classical increasing-only
LMI becomes infeasible at $s\geq 10$, while the combined LMI remains
feasible with moderate gains across the full range.

\begin{table}[ht]
  \centering
  \caption{Case~1 --- Direct output, H\"older $f$, increasing-only LMI.
           Classical LMI infeasible for $s\geq 10$; combined LMI
           remains feasible.}
  \label{tab:case1}
  \begin{tabular}{rrrrr}
    \toprule
    $s$ & $\|W_s\|$ & $\|K_{\mathrm{std}}\|$ & $\|K_{\mathrm{comb}}\|$ & $\|K'\|$ \\
    \midrule
    1.0 &   4.4 & 1925  &  36.3 &  15.6 \\
    5.0 &  22.0 & 1943  &  37.5 &  32.1 \\
   10.0 &  44.0 & ---   &  43.3 &  56.0 \\
   15.0 &  66.0 & ---   &  52.7 &  80.6 \\
   50.0 & 220.0 & ---   & 176.5 & 268.7 \\
    \bottomrule
  \end{tabular}
\end{table}

Figure~\ref{fig:case1_noise} shows the noise sweep at $s=10$; the combined observer
tracks the state accurately despite the absence of any slope bound,
achieving $\mathrm{RMS} \approx 0.18$ at $\sigma = 1$.

\begin{figure}[ht]
  \centering
   \includegraphics[width=\columnwidth]{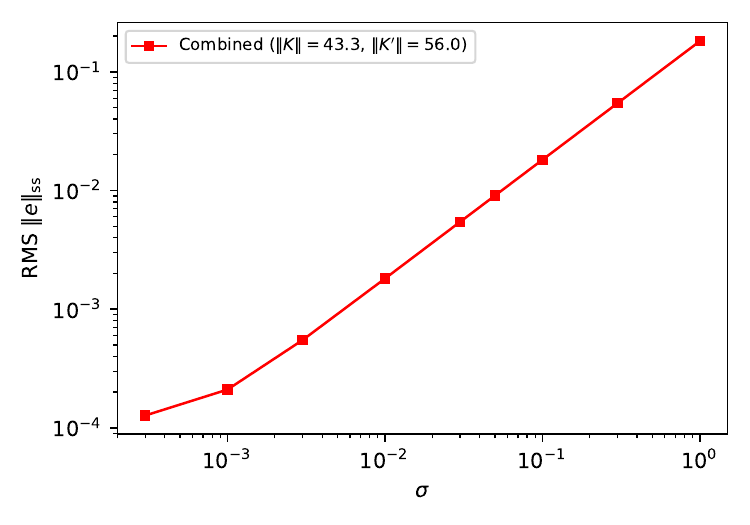}
  \caption{Case~1 --- Noise sweep at $s=10$ ($\|W\|=44.0$).
           No homogeneous observers; $Cf(x)=f(y)$ directly.  The classical LMI
           is infeasible; the combined observer converges.}
  \label{fig:case1_noise}
\end{figure}

\begin{figure}[ht]
  \centering
    \includegraphics[width=\columnwidth]{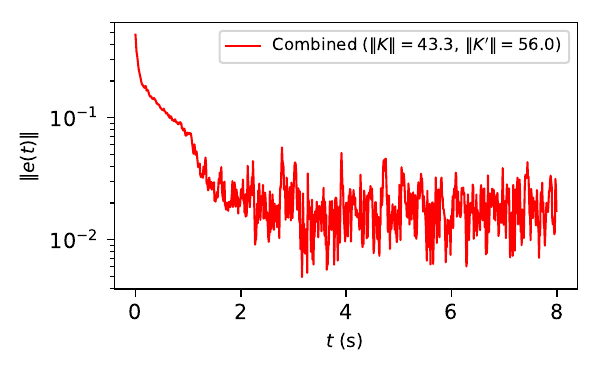}
  \caption{Case~1: Estimation error $\|e(t)\|$ at $s=10$,
           $\sigma=0.1$.  Classical LMI infeasible; combined observer
           converges.}
  \label{fig:case1_error}
\end{figure}

\subsection{Case 2: H\"older nonlinearity with reconstruction}
\label{sec:case2}

We now consider the LFP-type measurement matrix introduced above,
for which $Cf(x)$ is not directly available from $y$.
The family of homogeneous observers of
Proposition~\ref{prop:sliding} reconstructs $CWf(x)$ (the WC form with
$A=-a I$, Corollary~\ref{cor:outlin}).  The nonlinearity is the same H\"older
function as in Case~1, so no sector bound exists and only the
increasing-only LMI applies; Remark~\ref{rem:holder} covers the
boundedness required by Proposition~\ref{prop:sliding}.

Table~\ref{tab:case2} reports the gain norms.  The classical observer requires
gains that grow from $1763$ to $30701$, while the combined observer
maintains $\|K_{\mathrm{comb}}\|$ below $30$ and $\|K'\|$ stays constant at
$2.3$---the coupling attenuation is purely geometric.

\begin{table}[ht]
  \centering
  \caption{Case~2 --- H\"older $f$ + reconstruction of $CWf(x)$,
           increasing-only LMI.
           $\|K'\|$ constant; $\|K_{\mathrm{std}}\|$ explodes.}
  \label{tab:case2}
  \begin{tabular}{rrrrr}
    \toprule
    $s$ & $\|W_s\|$ & $\|K_{\mathrm{std}}\|$ & $\|K_{\mathrm{comb}}\|$ & $\|K'\|$ \\
    \midrule
    0.5 &   2.2 &  1763 &   6.3 & 2.3 \\
    1.0 &   4.4 &  1889 &   4.0 & 2.3 \\
    5.0 &  22.0 &  2178 &   8.2 & 2.3 \\
   10.0 &  44.0 &  3277 &  10.5 & 2.3 \\
   20.0 &  88.0 &  5368 &  21.5 & 2.3 \\
   50.0 & 220.0 & 30701 &  29.1 & 2.3 \\
    \bottomrule
  \end{tabular}
\end{table}

Figure~\ref{fig:case2_noise} shows the noise sweep at $s=10$.  At very low
noise ($\sigma \leq 3\cdot10^{-4}$), the reconstruction error
dominates and the classical observer has a slightly lower error floor.
Above $\sigma \approx 3\cdot10^{-3}$, the classical gain explosion
overtakes that penalty: the combined observer achieves a
$13\times$ lower RMS error at $\sigma=1$.
This ordering is the one predicted by Proposition~\ref{prop:iss}: although
the reconstruction term $\propto\sigma^{1/2}$ always dominates the linear
term $\propto\sigma$ for $\sigma<1$, its coefficient
$\|K'\|L_j\approx 6.9$ is smaller than
$\|K_{\mathrm{std}}\|=3277$ by a factor $475$, so the classical gain amplification
takes over at a noise level that scales as
$(\|K'\|L_j/\|K_{\mathrm{std}}\|)^2$.

\begin{figure}[ht]
  \centering
    \includegraphics[width=\columnwidth]{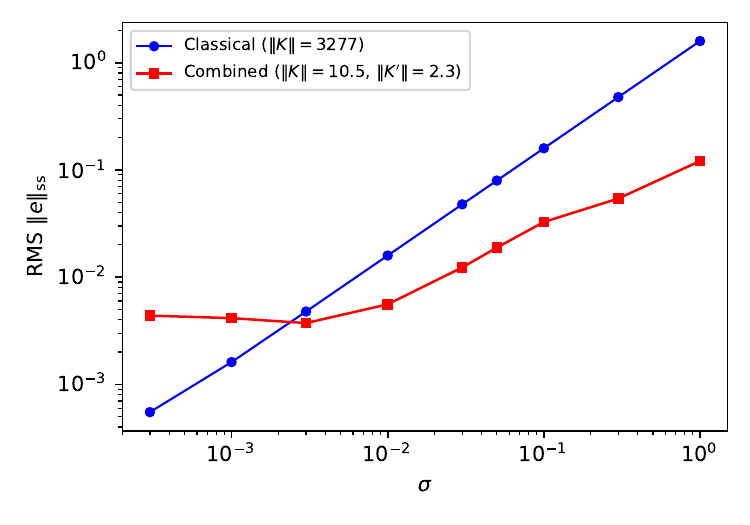}
  \caption{Case~2 --- Noise sweep at $s=10$ ($\|W\|=44.0$).
           The homogeneous observers reconstruct $CWf(x)$.
           Ratio reaches $13\times$ at $\sigma=1$.}
  \label{fig:case2_noise}
\end{figure}

\begin{figure}[ht]
  \centering
    \includegraphics[width=\columnwidth]{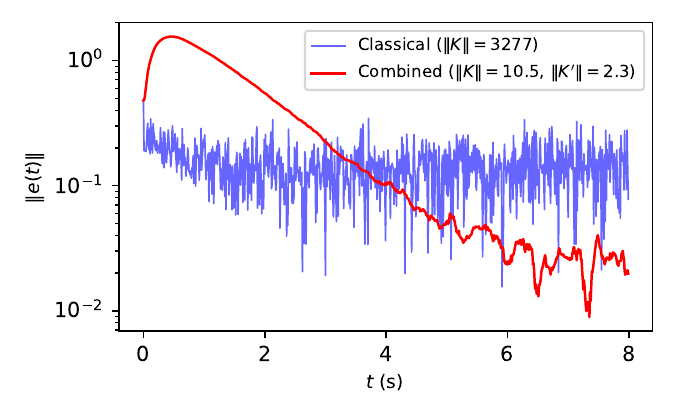}
  \caption{Case~2: Estimation error $\|e(t)\|$ at $s=10$,
           $\sigma=0.1$.  Combined observer ($\|K_{\mathrm{comb}}\|=10.5$)
           vs.\ classical ($\|K_{\mathrm{std}}\|=3277$).
           The reconstruction converges within $t\in[0,1]$,
           producing a transient overshoot; thereafter the combined
           observer settles at a lower error floor.}
  \label{fig:case2_error}
\end{figure}

\subsection{Case 3: Sigmoid with reconstruction}
\label{sec:case3}

We return to the original Wilson--Cowan configuration with the sigmoid
nonlinearity $f(\xi)=1/(1+e^{-\lambda_{\mathrm{sig}}\xi})$, $\lambda_{\mathrm{sig}}=4$,
and the LFP measurement matrix $C$.  The tight sector bound
$\Gamma = (\lambda_{\mathrm{sig}}/4)I_6 = I_6$ is used in the LMI.
The homogeneous observers reconstruct $CWf(x)$ as in Case~2.

Table~\ref{tab:case3} reports the gain norms.  Both designs attain the
maximum convergence rate $q=10.00$.  $\|K_{\mathrm{std}}\|$ grows from $19$
to $1271$, while $\|K_{\mathrm{comb}}\|$ grows from $5.7$ to $50$ and
$\|K'\|$ stays in $[2.3,\,3.1]$.

\begin{table}[ht]
  \centering
  \caption{Case~3 --- Lipschitz sigmoid + reconstruction of $CWf(x)$,
           sector LMI.
           $\|K'\|$ nearly constant; $\|K_{\mathrm{std}}\|$ grows $65\times$.}
  \label{tab:case3}
  \begin{tabular}{rrrrr}
    \toprule
    $s$ & $\|W_s\|$ & $\|K_{\mathrm{std}}\|$ & $\|K_{\mathrm{comb}}\|$ & $\|K'\|$ \\
    \midrule
    0.5 &   2.2 &    19 &   5.7 & 3.1 \\
    1.0 &   4.4 &    31 &   6.1 & 2.7 \\
    5.0 &  22.0 &    95 &  10.4 & 2.4 \\
   10.0 &  44.0 &   156 &  13.6 & 2.4 \\
   15.0 &  66.0 &   219 &  16.6 & 2.4 \\
   20.0 &  88.0 &   283 &  19.5 & 2.3 \\
   50.0 & 220.0 &   666 &  33.3 & 2.3 \\
  100.0 & 440.0 &  1271 &  49.9 & 2.3 \\
    \bottomrule
  \end{tabular}
\end{table}

Figure~\ref{fig:case3_noise} shows the noise sweep at $s=15$.
Below $\sigma\approx 2\cdot10^{-2}$ the reconstruction error
$\|K'\|\bar{\varepsilon}_2$ dominates and the classical observer has a
lower error floor.
Above this crossover, the $\|K_{\mathrm{std}}\|\sigma$ term takes over:
the $13\times$ smaller gain of the combined
observer translates into a $3.8\times$ lower RMS error at $\sigma=1$.
The ordering again follows Proposition~\ref{prop:iss}: while the
$\sigma^{1/2}$ reconstruction term dominates at low $\sigma$, its
coefficient $\|K'\|L_j\approx 6.9$ is very small compared to
$\|K_{\mathrm{std}}\|=219$ (Table~\ref{tab:case3}, $s=15$), so the classical
observer's linear noise amplification overtakes it.

\begin{figure}[ht]
  \centering
    \includegraphics[width=\columnwidth]{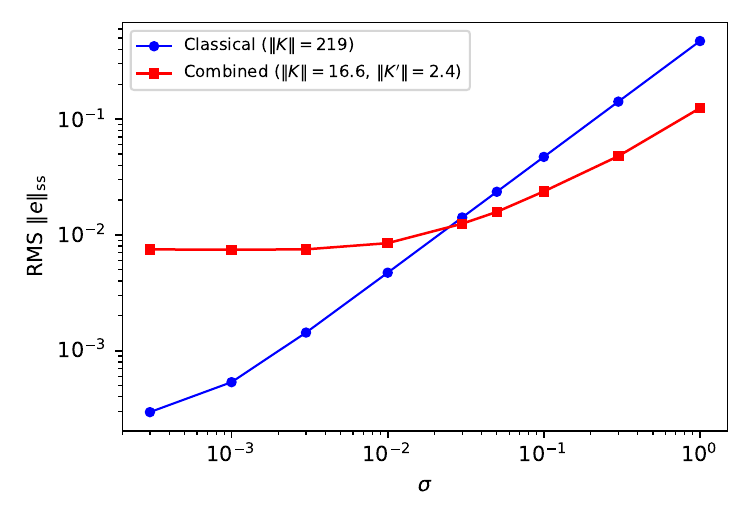}
  \caption{Case~3 --- Noise sweep at $s=15$ ($\|W\|=66.0$).
           Crossover near $\sigma\approx 2\cdot10^{-2}$;
           ratio $3.8\times$ at $\sigma=1$.}
  \label{fig:case3_noise}
\end{figure}

\begin{figure}[ht]
  \centering
    \includegraphics[width=\columnwidth]{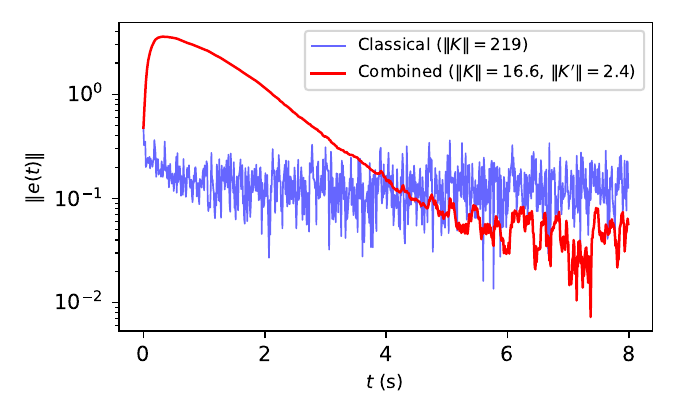}
  \caption{Case~3: Estimation error $\|e(t)\|$ at $s=15$,
           $\sigma=0.3$ (above the crossover).
           Combined observer ($\|K_{\mathrm{comb}}\|=16.6$) vs.\ classical
           ($\|K_{\mathrm{std}}\|=219$).
           As in Case~2, the convergence of the reconstruction
           produces a transient overshoot; the combined
           observer then settles at a lower steady-state error.}
  \label{fig:case3_error}
\end{figure}

\section{Conclusion}
\label{sec:conclusion}

We have proposed a combined observer for Lur'e-type systems with a general
sector-bounded nonlinearity that augments the standard linear output injection
with a nonlinear injection based on the finite-time reconstructed signal
$y_2 = CAx + CWf(x)$.  The key contributions are:
\begin{enumerate}[(i)]
  \item an LMI~\eqref{eq:LMI} containing the classical one as a special
        case that reshapes the coupling term via $K'$, keeping the
        observer gains moderate at prescribed convergence rates even
        when the classical design becomes gain-explosive
        (Theorem~\ref{thm:main}, Proposition~\ref{prop:general});
  \item an increasing-only LMI (Remark~\ref{rem:increasing}) that needs
        no sector slope bound and stays feasible for the combined
        observer when the classical LMI is infeasible;
  \item structural conditions under which the reconstruction simplifies:
        $CA = MC$ (Corollary~\ref{cor:outlin}) and
        direct availability of $Cf(x)$ from $y$
        (Proposition~\ref{prop:direct});
  \item a stability analysis (Proposition~\ref{prop:iss}) showing the
        linear-noise amplification scales with the moderate
        $\|K_{\mathrm{comb}}\|$ rather than the coupling-dependent
        $\|K_{\mathrm{std}}\|$, at the cost of an additional
        reconstruction error, and characterising the noise crossover
        below which the classical design is preferable;
  \item numerical confirmation on three scenarios of increasing
        complexity (Sections~\ref{sec:case1}--\ref{sec:case3}): the
        combined observer remains feasible where the classical
        increasing-only LMI becomes infeasible, and attains a $13\times$
        lower RMS error with sliding-mode reconstruction
        ($\|K'\| = 2.3$) and a $3.8\times$ lower error on the Lipschitz
        sigmoid model, where the noise crossover is quantified.
\end{enumerate}

\bibliographystyle{elsarticle-num}
\bibliography{refs}

\end{document}